\documentclass{amsart}
\usepackage{amssymb,latexsym,pstricks}
\topskip  \numberwithin{equation}{section}
\newtheorem{theorem}{Theorem}[section]

\newtheorem{corollary}[theorem]{Corollary}

\newtheorem{identity}[theorem]{Identity}

\begin{document}

\pagenumbering{arabic}
\pagestyle{headings}
\def\sof{\hfill\rule{2mm}{2mm}}
\def\llim{\lim_{n\rightarrow\infty}}

\title{Combinatorial Identities for Incomplete Tribonacci Polynomials}
\maketitle

\begin{center}
\author{Mark Shattuck\\
\small Department of Mathematics, University of Tennessee, Knoxville, TN 37996\\[-0.8ex]
\small\texttt{shattuck@math.utk.edu}\\[1.8ex]}
\end{center}

\section*{abstract}

The incomplete tribonacci polynomials, denoted by $T_n^{(s)}(x)$, generalize the usual tribonacci polynomials and were introduced in \cite{JR2}, where several algebraic identities were shown.  In this paper, we provide a combinatorial interpretation for $T_n^{(s)}(x)$ in terms of weighted linear tilings involving three types of tiles.  This allows one not only to supply combinatorial proofs of the identities for $T_n^{(s)}(x)$ appearing in \cite{JR2} but also to derive additional identities.  In the final section, we provide a formula for the ordinary generating function of the sequence $T_n^{(s)}(x)$ for a fixed $s$, which was requested in \cite{JR2}.  Our derivation is combinatorial in nature and makes use of an identity relating $T_n^{(s)}(x)$ to $T_n(x)$.\\

\noindent{Keywords}: tribonacci numbers, incomplete tribonacci polynomials, combinatorial proof

\noindent{2010 Mathematics Subject Classification}: 05A19, 05A15

\section{Introduction}

The \emph{tribonacci numbers} $t_n$ are defined by the recurrence relation $t_n=t_{n-1}+t_{n-2}+t_{n-3}$ if $n \geq 3$, with initial values $t_0=0$ and $t_1=t_2=1$.  See sequence A000073 in OEIS \cite{Sl}.  The tribonacci numbers are given equivalently by the explicit formula
\begin{equation}\label{ine1}
t_{n+1}=\sum_{i=0}^{\left\lfloor\frac{n}{2}\right\rfloor}B(n-i,i), \qquad n \geq 0,
\end{equation}
where $B(n,i)=\sum_{j=0}^i \binom{i}{j}\binom{n-j}{i}$, as shown in \cite{B}.  The number $B(n,i)$ is the $n$-th row, $i$-th column entry of the \emph{tribonacci triangle} (see \cite{AH}).

The \emph{tribonacci polynomials} $T_n(x)$ were introduced in \cite{HB} and are defined by the recurrence $T_n(x)=x^2T_{n-1}(x)+xT_{n-2}(x)+T_{n-3}(x)$ if $n \geq 3$, with initial values $T_0(x)=0$, $T_1(x)=1$, and $T_2(x)=x^2$.  In analogy to \eqref{ine1}, the tribonacci polynomials are given by the following explicit formula (see \cite{JR2}):
\begin{equation}\label{ine2}
T_{n+1}(x)=\sum_{i=0}^{\left\lfloor\frac{n}{2}\right\rfloor}\sum_{j=0}^i \binom{i}{j}\binom{n-i-j}{i}x^{2n-3(i+j)}.
\end{equation}

The \emph{incomplete tribonacci polynomials} $T_n^{(s)}(x)$ were considered in \cite{JR2} and are defined as
\begin{equation}\label{ine3}
T_{n+1}^{(s)}(x)=\sum_{i=0}^{s}\sum_{j=0}^i \binom{i}{j}\binom{n-i-j}{i}x^{2n-3(i+j)}, \qquad 0 \leq s \leq \left\lfloor\frac{n}{2}\right\rfloor.
\end{equation}
Note that the incomplete tribonacci polynomials generalize the ordinary ones and reduce to them when $s=\left\lfloor\frac{n}{2}\right\rfloor$. The \emph{incomplete tribonacci number}, denoted by $t_n^{(s)}$, is defined as the value of $T_n^{(s)}(x)$ at $x=1$.  Incomplete Fibonacci numbers and polynomials have also been considered and are defined in a comparable way; see, e.g., \cite{PF,JR1}.  Some combinatorial identities for the incomplete Fibonacci numbers were given in \cite{BB} and a bi-periodic generalization was studied in \cite{JR3}.

In \cite{JR2}, several identities were derived for the incomplete tribonacci
numbers and polynomials using various algebraic methods.  In this paper, we supply combinatorial proofs of these identities using a weighted tiling interpretation of $T_n^{(s)}(x)$ (described in Theorem \ref{t1} below). In some cases, a further generalization of an identity can be given.  In addition, using our interpretation, one also can find other relations not given in \cite{JR2} that are satisfied by $T_n^{(s)}(x)$.  In the final section, we provide  an explicit formula for the generating function of $T_n^{(s)}(x)$, as requested in \cite{JR2}.  Our derivation is combinatorial in nature and makes use of some identities involving $T_n(x)$.

\section{Combinatorial interpretation for $T_{n}^{(s)}(x)$}

We will use the following terminology.  By a \emph{square}, \emph{domino}, or \emph{tromino}, we will mean, respectively, a $1 \times 1$, $2 \times 1$, or $3 \times 1$ rectangular tile.  A (linear) \emph{tiling} of length $n$ is a covering of the numbers $1,2,\ldots,n$ written in a row by squares, dominos, and trominos, where tiles of the same kind are indistinguishable.  Let $\mathcal{T}_n$ denote the set of all tilings of length $n$.  It is well known that $\mathcal{T}_n$ has cardinality $t_{n+1}$ (see, e.g., \cite[p. 36]{BQ}).  We will often represent squares, dominos, and trominos by the letters $r$, $d$, and $t$, respectively.  Thus, a member of $\mathcal{T}_n$ may be regarded as a word in the alphabet $\{r,d,t\}$ in which there are $n-2i-3j$, $i$, and $j$ occurrences of the letters $r$, $d$, and $t$, respectively, for some $i$ and $j$.

By a \emph{longer piece} within a member of $\mathcal{T}_n$, we will mean one that is either a domino or a tromino.  Given $0 \leq s \leq \left\lfloor\frac{n}{2}\right\rfloor$, let $\mathcal{T}_{n}^{(s)}$ denote the subset of $\mathcal{T}_n$ whose members contain \emph{at most} $s$ longer pieces.  For example, if $n=5$ and $s=1$, then $\mathcal{T}_5^{(1)}=\{r^5,dr^3,rdr^2,r^2dr,r^3d,tr^2,rtr,r^2t\}$. Note that $\mathcal{T}_n^{(s)}$ is all of $\mathcal{T}_n$ when $s=\left\lfloor\frac{n}{2}\right\rfloor$.  By a \emph{square-and-domino} tiling, we will mean a member of $\mathcal{T}_n$ that contains no trominos.

Given $\pi \in \mathcal{T}_n^{(s)}$, let $\delta(\pi)$ and $\nu(\pi)$ record the number of squares and dominos, respectively, in $\pi$.  We now provide a combinatorial interpretation of $T_{n+1}^{(s)}(x)$ in terms of linear tilings.

\begin{theorem}\label{t1}
The polynomial $T_{n+1}^{(s)}(x)$ is the distribution for the statistic $2\delta+\mu$ on $\mathcal{T}_n^{(s)}$.
\end{theorem}
\begin{proof}
First note that $T_{n+1}^{(s)}(x)$ may be written as
\begin{equation}\label{e1}
T_{n+1}^{(s)}(x)=\sum_{i=0}^s B(n-i,i)(x),
\end{equation}
where $B(n,i)(x)=\sum_{j=0}^i \binom{i}{j}\binom{n-j}{i}x^{2n-i-3j}$.  We next observe that when $x=1$, the polynomial $B(n,i)(x)$ gives the cardinality of the set $\mathcal{B}_{n,i}$ consisting of square-and-domino tilings of length $n$ in which the squares come in two colors, black and white, and containing $i$ dominos and white squares combined.  To see this, note that members of $\mathcal{B}_{n,i}$ containing exactly $j$ dominos are in one-to-one correspondence with words in the alphabet $\{D,W,B\}$ containing $j$ $D$'s, $i-j$ $W$'s, and $n-i-j$ $B$'s and thus have cardinality
$$\binom{n-j}{j,i-j,n-i-j}=\frac{(n-j)!}{j!(i-j)!(n-i-j)!}=\binom{n-j}{i}\binom{i}{j}.$$
Summing over $j$ gives
$$|\mathcal{B}_{n,i}|=\sum_{j=0}^i \binom{i}{j}\binom{n-j}{i}.$$

Given $\pi \in \mathcal{B}_{n,i}$, let $\delta_1(\pi)$ and $\delta_2(\pi)$ record the number of black and white squares, respectively. Thus, if $\pi \in \mathcal{B}_{n,i}$ has $j$ dominos, then
$$2\delta_1(\pi)+\delta_2(\pi)=2(n-i-j)+i-j=2n-i-3j.$$
Considering all $j$, this implies $B(n,i)(x)$ is the distribution on $\mathcal{B}_{n,i}$ for the statistic $2\delta_1(\pi)+\delta_2(\pi)$.
Suppose now $\lambda \in \mathcal{B}_{n-i,i}$ is given and contains $j$ dominos for some $j$, where $0 \leq i \leq s$.  We replace each domino of $\lambda$ with a tromino and each white square with a domino.  The resulting tiling $\lambda'$ belongs to $\mathcal{T}_n^{(s)}$ and has $j$ trominos, $i-j$ dominos, and $n-2i-j$ squares.  Thus we have
$$2\delta(\lambda')+\nu(\lambda')=2\delta_1(\lambda)+\delta_2(\lambda)$$
for all $\lambda \in \mathcal{B}_{n-i,i}$.  By \eqref{e1}, it follows that $T_{n+1}^{(s)}(x)$ is the distribution on $\cup_{i=0}^s \mathcal{B}_{n-i,i}$ for $2\delta_1+\delta_2$, equivalently, for the distribution of $2\delta+\nu$ on $\mathcal{T}_n^{(s)}$.
\end{proof}

\emph{Remark:} Taking $x=1$ in the prior theorem shows that the cardinality of $\mathcal{T}_n^{(s)}$ is $t_{n+1}^{(s)}$.  Taking $s=\left\lfloor\frac{n}{2}\right\rfloor$ shows that $T_{n+1}(x)$ is the distribution for $2\delta+\mu$ on all of $\mathcal{T}_n$.

Using our interpretation for $T_{n}^{(s)}(x)$, one obtains the following recurrence formula from \cite{JR2} as a corollary.

\begin{corollary}\label{t1c1}
If $n \geq 2s+1$, then
\begin{equation}\label{c1e1}
T_{n+3}^{(s)}(x)=x^2T_{n+2}^{(s)}(x)+xT_{n+1}^{(s)}(x)+T_n^{(s)}(x)-(xB(n-s,s)(x)+B(n-1-s,s)(x)).
\end{equation}
\end{corollary}
\begin{proof}
We will show that the right-hand side of \eqref{c1e1} gives the weighted sum of all the members of $\mathcal{T}_{n+2}^{(s)}$ with respect to the statistic $2\delta+\nu$ by considering the final piece.  The first term clearly accounts for all tilings ending in a square.  On the other hand, if a member of $\mathcal{T}_{n+2}^{(s)}$ ends in a longer piece, then there can be at most $s-1$ additional longer pieces.  From the proof of Theorem \ref{t1} above, we have for each $m$ that $B(m-s,s)(x)$ gives the weight of all members of $\mathcal{T}_m^{(s)}$ containing exactly $s$ longer pieces. Note that addition of a longer piece to the end of a tiling already containing $s$ longer pieces is not allowed.  Thus, by subtraction, the total weight of all members of $\mathcal{T}_{n+2}^{(s)}$ ending in a domino is given by $x(T_{n+1}^{(s)}(x)-B(n-s,s)(x))$ and the weight of those ending in a tromino by $T_{n}^{(s)}(x)-B(n-1-s,s)(x)$, which completes the proof.
\end{proof}

\section{Some identities of $T_{n}^{(s)}(x)$}

In this section, we provide combinatorial proofs of some identities involving the incomplete tribonacci polynomials that generalize those shown in \cite{JR2} using algebraic methods.  We also consider some further identities that can be obtained from the combinatorial interpretation given in Theorem \ref{t1}.

In this section and the next, by the \emph{weight} of a subset $S$ of $\mathcal{T}_n$ or $\mathcal{T}_n^{(s)}$, we will mean the sum $\sum_{\lambda\in S} x^{2\delta(\lambda)+\nu(\lambda)}$.

The $x=1$ case of the following identity was shown in \cite{JR2} by an inductive argument.

\begin{identity}\label{i1}
If $h \geq 1$ and $n \geq 2s+2$, then
\begin{equation}\label{i1e1}
\sum_{i=0}^{h-1}x^{2(h-i-1)}T_{n+i}^{(s)}(x)=\frac{1}{1+x^3}\left(T_{n+h+2}^{(s+1)}(x)-x^{2h}T_{n+2}^{(s+1)}(x)+x^{2h+1}T_n^{(s)}(x)-xT_{n+h}^{(s)}(x)\right).
\end{equation}
\end{identity}
\begin{proof}
We show equivalently
$$T_{n+h+2}^{(s+1)}(x)=(1+x^3)\sum_{i=0}^{h-1}x^{2(h-i-1)}T_{n+i}^{(s)}(x)+x^{2h}T_{n+2}^{(s+1)}(x)+xT_{n+h}^{(s)}(x)-x^{2h+1}T_n^{(s)}(x).$$
For this, we'll argue that the right-hand side gives the total weight of all the members of $\mathcal{T}_{n+h+1}^{(s+1)}$.  First note that $x^{2h}T_{n+2}^{(s+1)}(x)$ gives the weight of the members of $\mathcal{T}_{n+h+1}^{(s+1)}$ in which positions $n+2$ through $n+h+1$ are covered by squares (i.e., the right-most longer piece ends at position $n+1$ or before).  On the other hand, the weight of all members of $\mathcal{T}_{n+h+1}^{(s+1)}$ whose right-most longer piece starts at position $n+i-1$ for some $0 \leq i \leq h-1$ is given by $\left(x^{2(h-i)+1}+x^{2(h-i-1)}\right)T_{n+i}^{(s)}(x)$ since such tilings $\lambda$ are of the form $\lambda=\lambda'dr^{h-i}$ or $\lambda=\lambda'tr^{h-i-1}$ for some tiling $\lambda'$ of length $n+i-1$ where $r^m$ denotes a sequence of $m$ squares. Note that $\lambda' \in \mathcal{T}_{n+i-1}^{(s)}$ since the number of longer pieces in $\lambda'$ is limited to $s$.  Summing over $0 \leq i \leq h-1$ gives the indexed sum on the right-hand side.  Next, the term $xT_{n+h}^{(s)}(x)$ accounts for all members of $\mathcal{T}_{n+h+1}^{(s+1)}$ whose final piece is a domino which were missed in the sum.  Finally, members of $\mathcal{T}_{n+h+1}^{(s+1)}$  of the form $\lambda'dr^h$, where $\lambda'$ has length $n-1$, were accounted for by both the $x^{2h}T_{n+2}^{(s+1)}(x)$ term and by the $i=0$ term of the indexed sum; hence, we must subtract their weight, $x^{2h+1}T_n^{(s)}(x)$, to correct for this double count.  Combining all of the cases above completes the proof.
\end{proof}

The following identity from \cite{JR2} gives a formula for the sum of all the incomplete tribonacci polynomials of a fixed order.

\begin{identity}\label{i2}
If $n \geq 1$, then
\begin{equation}\label{i2e1}
\sum_{s=0}^\ell T_{n+1}^{(s)}(x)=(\ell+1)T_{n+1}(x)-\sum_{i=0}^\ell \sum_{j=0}^i i\binom{i}{j}\binom{n-i-j}{i}x^{2n-3(i+j)},
\end{equation}
where $\ell=\left\lfloor\frac{n}{2}\right\rfloor$.
\end{identity}
\begin{proof}
Let $\lambda \in \mathcal{T}_n$ and suppose that it contains exactly $k$ longer pieces, where $0 \leq k \leq \ell$.  Then the weight of $\lambda$ is counted by each summand of $\sum_{s=0}^\ell T_{n+1}^{(s)}(x)$ such that $s \geq k$.  That is, the tiling $\lambda$ is counted $\ell+1-k$ times by this sum.  The proof of Theorem \ref{t1}
above shows that the total weight of all members of $\mathcal{T}_n$ containing exactly $k$ longer pieces is given by
$$\sum_{j=0}^k \binom{k}{j}\binom{n-k-j}{k}x^{2n-3(k+j)},$$
upon considering the number $j$ of dominos.  Thus, the only inner sum in the double sum on the right-hand side of \eqref{i2e1} in which $\lambda$ is counted occurs when $i=k$ and here it is counted $k$ times (due to the extra factor of $i=k$).  Since $\lambda$ is clearly counted $\ell+1$ times by the term $(\ell+1)T_{n+1}(x)$, we have by subtraction that $\lambda$ is counted $\ell+1-k$ times by the right-hand side of \eqref{i2e1} as well.  Since $\lambda$ was arbitrary, the identity follows.
\end{proof}

The $x=1$ case of the following identity was conjectured in \cite{JR2} and follows from the generating function proof given in \cite{KP}.  

\begin{identity}\label{i3}
If $n \geq 1$, then
\begin{equation}\label{i3e1}
\sum_{s=0}^{\ell} T_{n+1}^{(s)}(x)=(\ell+1)T_{n+1}(x)-\sum_{j=1}^{n-1}(xT_j(x)+T_{j-1}(x))T_{n-j}(x),
\end{equation}
\end{identity}
where $\ell=\left\lfloor\frac{n}{2}\right\rfloor$.
\begin{proof}
Suppose $\lambda \in \mathcal{T}_n$ has exactly $k$ longer pieces.  By the proof of the preceding identity, we need only show that the weight of $\lambda$ is counted $k$ times by the sum on the right-hand side of \eqref{i3e1}.  Note that $xT_j(x)T_{n-j}(x)$ gives the weight of all members of $\mathcal{T}_n$ in which a domino covers positions $j$ and $j+1$, while $T_{j-1}(x)T_{n-j}(x)$ gives the weight of those in which a tromino covers positions $j-1$, $j$, and $j+1$.  Thus, for each longer piece of $\lambda$, there is a term in the sum that counts the weight of $\lambda$, which implies that $\lambda$ is counted $k$ times by the sum, as desired.
\end{proof}

\emph{Remark:}  Comparing the $x=1$ cases of the preceding two identities, it follows that
$$\sum_{n\geq 1}a_nz^n=\frac{z^2+z^3}{(1-z-z^2-z^3)^2},$$
where
$$a_n=\sum_{i=0}^{\left\lfloor\frac{n}{2}\right\rfloor} \sum_{j=0}^i i\binom{i}{j}\binom{n-i-j}{i},$$
which can also be shown directly using the methods of \cite[Section 4.3]{W} (see \cite{KP}).

The next three identities follow from the combinatorial interpretation of $T_n^{(s)}(x)$ given in Theorem \ref{t1} and do not occur in \cite{JR2}.

\begin{identity}\label{i4}
If $n \geq 2s+1$, then
\begin{equation}\label{i4e1}
T_{n+1}^{(s)}(x)=\sum_{i=0}^s(x^{i+2}T_{n-2i}^{(s-i)}(x)+x^iT_{n-2i-2}^{(s-i-1)}(x)).
\end{equation}
\end{identity}
\begin{proof}
Suppose a member of $\mathcal{T}_{n}^{(s)}$ ends in exactly $i$ dominos, where $0 \leq i \leq s$.  If the right-most piece that is not a domino is a square, then the tiles coming to the left of this square constitute a member of $\mathcal{T}_{n-2i-1}^{(s-i)}$ and thus the weight of the corresponding subset of $\mathcal{T}_{n}^{(s)}$ is $x^{i+2}T_{n-2i}^{(s-i)}(x)$.  On the other hand, if the right-most non-domino piece is a tromino, then the tiles to the left of this tromino form a member of $\mathcal{T}_{n-2i-3}^{(s-i-1)}$ and thus the weight of the corresponding subset is $x^iT_{n-2i-2}^{(s-i-1)}(x)$.  Considering all possible $i$ gives \eqref{i4e1}.
\end{proof}

Our next formula relates the incomplete tribonacci polynomials to the trinomial coefficients.

\begin{identity}\label{i5}
If $n \geq 3s+1$, then
\begin{equation}\label{i5e1}
T_n^{(s)}(x)=\sum_{i=0}^s \sum_{j=0}^{s-i} \binom{s}{i,j,s-i-j}x^{2s-i-2j}T_{n-s-i-2j}^{(s-i-j)}(x).
\end{equation}
\end{identity}
\begin{proof}
Suppose that there are $i$ dominos and $j$ trominos among the final $s$ tiles within a member of $\mathcal{T}_{n-1}^{(s)}$, where $n \geq 3s+1$.  Then there are $\binom{s}{i,j,s-i-j}$ ways to arrange these tiles, which contribute $x^{2(s-i-j)+i}$ towards the weight, with the remaining tiles forming a member of $\mathcal{T}_{n-s-i-2j-1}^{(s-i-j)}$.  Considering all possible $i$ and $j$ gives \eqref{i5e1}.
\end{proof}

The incomplete Fibonacci polynomials introduced in \cite{JR1} are given as
$$F_n^{(s)}(x)=\sum_{r=0}^s \binom{n-r-1}{r}x^{n-2r-1}, \qquad 0 \leq s \leq \left\lfloor\frac{n-1}{2}\right\rfloor.$$
Our next identity relates the incomplete Fibonacci and tribonacci polynomials.

\begin{identity}\label{i6}
If $n \geq 2s$, then
\begin{equation}\label{i6e1}
T_{n+1}^{(s)}(x)=x^{n/2}F_{n+1}^{(s)}(x^{3/2})+\sum_{i=1}^{n-2}x^{(i-1)/2}\sum_{j=0}^{s-1} (T_{n-i-1}^{(j)}(x)-T_{n-i-1}^{(j-1)}(x))F_i^{(s-j-1)}(x^{3/2}).
\end{equation}
\end{identity}
\begin{proof}
First note that the weight of all members of $\mathcal{T}_n^{(s)}$ that contain no trominos is given by
$$\sum_{r=0}^s\binom{n-r}{r}x^{2n-3r}=x^{n/2}F_{n+1}^{(s)}(x^{3/2}).$$
So assume a member of $\mathcal{T}_n^{(s)}$ contains at least one tromino and that the left-most tromino covers positions $i$ through $i+2$.  Suppose further that there are exactly $r$ dominos to the left of the left-most tromino.  Then the weight of all such members of $\mathcal{T}_n^{(s)}$ is given by $\binom{i-r-1}{r}x^{2i-3r-2}T_{n-i-1}^{(s-r-1)}(x)$. Summing over the possible $i$ and $r$ implies that the total weight of all the members of $\mathcal{T}_n^{(s)}$ containing at least one tromino is
$$\sum_{i=1}^{n-2}\sum_{r=0}^{s-1}\binom{i-r-1}{r}x^{2i-3r-2}T_{n-i-1}^{(s-r-1)}(x).$$

To obtain the expression in \eqref{i6e1}, we write $T_{n-i-1}^{(s-r-1)}$ as $\sum_{j=0}^{s-r-1}(T_{n-i-1}^{(j)}-T_{n-i-1}^{(j-1)})$, where $T_{n-i-1}^{(-1)}=0$.  We then obtain a total weight formula of
\begin{align*}
&\sum_{i=1}^{n-2}\sum_{r=0}^{s-1} \binom{i-r-1}{r}x^{2i-3r-2}\sum_{j=0}^{s-r-1}(T_{n-i-1}^{(j)}-T_{n-i-1}^{(j-1)})\\
&=\sum_{i=1}^{n-2}\sum_{j=0}^{s-1}(T_{n-i-1}^{(j)}-T_{n-i-1}^{(j-1)})\sum_{r=0}^{s-j-1}\binom{i-r-1}{r}x^{2i-3r-2}\\
&=\sum_{i=1}^{n-2}\sum_{j=0}^{s-1}(T_{n-i-1}^{(j)}-T_{n-i-1}^{(j-1)})x^{(i-1)/2}F_i^{(s-j-1)}(x^{3/2}),
\end{align*}
which gives \eqref{i6e1}.
\end{proof}

\section{Generating function formula for $T_n^{(s)}(x)$}

The generating function formula for the incomplete tribonacci numbers was found in \cite{JR2} and a formula was requested for the tribonacci polynomials.  The next result provides such a formula.  We remark that our method is more combinatorial than that used in \cite{JR2} in the case $x=1$ and thus supplies an alternate proof in that case.

\begin{theorem}\label{t2}
Let $Q_s(z)$ be the generating function for the incomplete tribonacci polynomials $T_n^{(s)}(x)$, where $n \geq 2s+1$.  Then
\begin{equation}\label{t2e1}
\frac{Q_s(z)}{z^{2s+1}}=\frac{T_{2s+1}(x)+(T_{2s-1}(x)+xT_{2s}(x))z+T_{2s}(x)z^2-z^2\left(\frac{x+z}{1-x^2z}\right)^{s+1}}{1-x^2z-xz^2-z^3}.
\end{equation}
\end{theorem}
\begin{proof}
Let $r_n=r_n(x)$ be given by
$$r_n=\sum_{j=0}^s \binom{s}{j}\binom{n+s-j-2}{s}x^{2n+s-3j-3}+\sum_{j=0}^s\binom{s}{j}\binom{n+s-j-3}{s}x^{2n+s-3j-6}, \qquad n \geq 3,$$
with $r_0=r_1=0$ and $r_2=x^{s+1}$.

We claim that $r_i(x)$ gives the total weight with respect to the statistic $2\delta+\nu$ of all the members of $\mathcal{T}_{i+2s}$ containing exactly $s+1$ longer pieces and ending in a longer piece, the subset of which we will denote by $\mathcal{A}$.  To show this, first note that $r_i(x)$ evaluated at $x=1$ is seen to give the number of square-and-domino tilings of length $i+s-2$ or $i+s-3$ in which squares are black or white and having exactly $s$ white squares and dominos combined.  We then increase the length of each white square and each domino by one and add a domino to the end if the original tiling had length $i+s-2$ and add a tromino to the end if it had length $i+s-3$.  This yields all members of $\mathcal{A}$ in a bijective manner and thus implies $r_i(x)$ at $x=1$ gives the cardinality of $\mathcal{A}$.  Note that members of $\mathcal{A}$ ending in a domino contain $i-j-2$ squares, $s-j+1$ dominos, and $j$ trominos for some $0 \leq j \leq s$, while members of $\mathcal{A}$ ending is a tromino contain $i-j-3$ squares, $s-j$ dominos, and $j+1$ trominos for some $j$.  Summing over $j$ then implies that $r_i(x)$ is the distribution for the statistic $2\delta+\nu$ on $\mathcal{A}$, as claimed.

By the interpretation for $r_i(x)$ just described, the product $r_i(x)T_{n-2s-i}(x)$ gives the total weight of all members of $\mathcal{T}_{n-1}$ containing at least $s+1$ longer pieces in which the $(s+1)$-st longer piece ends at position $i+2s$ since the final $n-2s-i-1$ positions of such a member of $\mathcal{T}_{n-1}$ may be covered by any tiling of the same length.  Summing over all possible $i$ then gives the total weight of all members of $\mathcal{T}_{n-1}$ containing \emph{strictly more} than $s$ longer pieces.  Subtracting from $T_n(x)$ thus gives the weight of all members of $\mathcal{T}_{n-1}$ containing \emph{at most} $s$ longer pieces and implies the following identity:
\begin{equation}\label{t2e2}
T_n^{(s)}(x)=T_n(x)-\sum_{i=0}^{n-2s-1}r_i(x)T_{n-2s-i}(x), \qquad n \geq 2s+1.
\end{equation}

In order to find a closed form expression for $Q_s(z)$ using \eqref{t2e2}, we express $T_n=T_n(x)$ as follows:
\begin{equation}\label{t2e3}
T_n=T_{n-2s}T_{2s+1}+T_{n-2s-1}(T_{2s-1}+xT_{2s})+T_{n-2s-2}T_{2s}, \qquad n \geq 2s+1.
\end{equation}
We provide a combinatorial proof of \eqref{t2e3} as follows.  Note that \eqref{t2e3} is clearly true if $s=0$ or if $n=2s+1$ since $T_0=T_{-1}=0$, so we may assume $s\geq 1$ and $n \geq 2s+2$.  Observe first that the $T_{n-2s}T_{2s+1}$ term gives the weight of all members of $\mathcal{T}_{n-1}$ in which there is no piece covering the boundary between positions $2s$ and $2s+1$.  On the other hand, the total weight of the members of $\mathcal{T}_{n-1}$ in which a domino covers this boundary is given by $xT_{n-2s-1}T_{2s}$.  Finally, if a tromino covers the boundary between positions $2s$ and $2s+1$, then that tromino covers either positions $2s-1$, $2s$, and $2s+1$ or positions $2s$, $2s+1$, and $2s+2$.  In the former case, the weight of the corresponding members of $\mathcal{T}_{n-1}$ is $T_{n-2s-1}T_{2s-1}$, while in the latter it would be $T_{n-2s-2}T_{2s}$.  Combining all of the cases above gives \eqref{t2e3}.

Multiplying both sides of the equation
$$T_n^{(s)}=T_{n-2s}T_{2s+1}+T_{n-2s-1}(T_{2s-1}+xT_{2s})+T_{n-2s-2}T_{2s}-\sum_{i=0}^{n-2s-1}r_iT_{n-2s-i}$$
by $z^n$ and summing over $n \geq 2s+1$ yields
$$\frac{Q_s(z)}{z^{2s}}=\left(T_{2s+1}+(T_{2s-1}+xT_{2s})z+T_{2s}z^2-\sum_{i\geq0}r_iz^i\right)
\cdot\sum_{n\geq1}T_nz^n.$$
The proof is completed by noting
$$\sum_{i\geq0}r_iz^i=z^2\left(\frac{x+z}{1-x^2z}\right)^{s+1}$$
and
$$\sum_{n\geq1}T_nz^n=\frac{z}{1-x^2z-xz^2-z^3},$$
the former being computed by the methods given in \cite[Section 4.3]{W}.
\end{proof}

Taking $x=1$ in the prior theorem yields the following result.

\begin{corollary}\label{t2c1}
Let $q_s(z)$ be the generating function for the incomplete tribonacci numbers $t_n^{(s)}$.  Then
\begin{equation}\label{t2e1}
\frac{q_s(z)}{z^{2s+1}}=\frac{t_{2s+1}+(t_{2s-1}+t_{2s})z+t_{2s}z^2-z^2\left(\frac{1+z}{1-z}\right)^{s+1}}{1-z-z^2-z^3}.
\end{equation}
\end{corollary}

\emph{Remark:} Equation \eqref{t2e1} appears as Theorem 8 of \cite{JR2}.  We note however that there was a slight misstatement of this theorem; in particular, there should be no $-2$ in the factor multiplying $z^2$ in the numerator on the right-hand side.


\end{document}